\documentclass{amsart}
\usepackage{graphicx}
\usepackage{enumerate}
\usepackage{amssymb}
\usepackage{amsaddr}
\newtheorem{thm}{Theorem}[section]

\newtheorem{lem}[thm]{Lemma}
\newtheorem{prop}[thm]{Proposition}
\newtheorem{conj}[thm]{Conjecture}

\newtheorem{notn}[thm]{Notation}

\newtheorem{example}[thm]{Example}
\theoremstyle{definition}
\theoremstyle{remark}

\numberwithin{equation}{section}
\newcommand{\per}{\text{\textnormal{per}}}
\newcommand{\sgn}{\text{\textnormal{sgn}}}
\newcommand{\Text}[1]{\text{\textnormal{#1}}}
\begin{document}

\title{A weak version of Rota's basis conjecture for odd dimensions}
\setlength{\parskip}{0in}
\author{Ron Aharoni}%
\address{Department of Mathematics, Technion, Haifa 32000, Israel}%
\email{ra@tx.technion.ac.il}%
\author{Daniel Kotlar}%
\address{Computer Science Department, Tel-Hai College, Upper Galilee 12210, Israel}%
\email{dannykot@telhai.ac.il}%
\setlength{\parskip}{0.075in}

%\date{November 17, 2011}
%\dedicatory{}%
%\commby{}%
% ----------------------------------------------------------------
\begin{abstract}
The Alon-Tarsi Latin square conjecture is extended to odd dimensions by stating it for reduced Latin squares (Latin squares having the identity permutation as their first row and first column). A modified version of Onn's colorful determinantal identity is used to show how the validity of this conjecture implies a weak version of Rota's basis conjecture for odd dimensions, namely that a set of $n$ bases in $\mathbb{R}^n$ has $n-1$ disjoint independent transversals.
\end{abstract}
\maketitle
% ----------------------------------------------------------------
\section{Rota's basis conjecture and some conjectures on Latin squares}\label{section1}
Given a family of sets $B_1,\ldots,B_n$, a \emph{transversal} is a set that contains exactly one element from each of the given sets. In 1989 G.-C. Rota made the following conjecture \cite{HuangRota94}:
\begin{conj}[\textbf{Rota's basis conjecture}]\label{conjRota}
Let $B_1,B_2,\ldots,B_n$ be bases of an $n$-dimensional vector space over an arbitrary field. Then, their multiset union can be repartitioned into $n$ transversals that are all bases.
\end{conj}
The original conjecture was stated, in fact, in the more general setting of matroids. Some special cases of Conjecture~\ref{conjRota} were solved in \cite{AB06,Chan95,Chow95,Drisko97,Drisko98,Geelen06,Geelen07,Glynn10,HuangRota94,KZ05,Pon04,Wild94}.
Huang and Rota \cite{HuangRota94}, and independently Onn \cite{Onn97}, noticed a connection between the conjecture and number of even and odd Latin squares. Let $\mathcal{L}$ be the set of all Latin squares of size $n\times n$ over $\{1,\ldots,n\}$. For a Latin square $L\in\mathcal{L}$, we use the notation $L_i$ for its $i$th row and $L^j$ for its $j$th column. the \emph{sign}, or \emph{parity}, of $L$, denoted $\sgn(L)$, is defined as the product of the signs of all its row and column permutations, that is, $\sgn(L)=\prod_{i=1}^n\sgn(L_i)\sgn(L^i)$. A Latin square $L$ is \emph{even} if $\sgn(L)=1$, and \emph{odd} if $\sgn(L)=-1$. For a given dimension $n$, Let $els(n)$ denote the number of even Latin squares of order $n$, let $ols(n)$ denote the number of odd ones, and let $l(n)=els(n)-ols(n)$. For odd $n>1$ it is easy to see that $l(n)= 0$. For even $n$ and for a field of characteristic 0 Conjecture~\ref{conjRota} was shown in \cite{HuangRota94} and \cite{Onn97} to be a consequence of the following conjecture of Alon and Tarsi \cite{AlonTarsi92}:
\begin{conj}[\textbf{Alon-Tarsi Latin square conjecture}]\label{conjAlonTarsi}
For all even $n$, $l(n)\neq0$.
\end{conj}

 No general result for Conjecture~\ref{conjRota} for odd $n$ has been presented yet. In this paper an analogue of the relation between Conjectures~\ref{conjAlonTarsi} and \ref{conjRota} is shown for odd $n$.

 The following terms and notation appear in \cite{Janssen95} and \cite{Zappa97}, among others. A Latin square is said to be \emph{reduced} if its first row and first column are the identity permutation. A Latin square is said to be \emph{normalized} if its first row is the identity permutation. A Latin square is said to be \emph{diagonal} if its diagonal consists solely of 1's. Let $rels(n)$, $rols(n)$, $ndels(n)$ and $ndols(n)$ be the numbers of reduced even Latin squares, reduced odd Latin squares, normalized diagonal even Latin squares and normalized diagonal odd Latin squares, respectively. Zappa \cite{Zappa97} introduced $AT(n)=ndels(n)-ndols(n)$ and proposed the following extension to Conjecture~\ref{conjAlonTarsi}:
\begin{conj}[\textbf{Extended Alon-Tarsi conjecture}]\label{conjZappa}
$AT(n)\neq0$ for every positive $n$.
\end{conj}
For even $n$ this conjecture is equivalent to Conjecture~\ref{conjAlonTarsi}. For odd $n$, Drisko \cite{Drisko98} proved the conjecture in the case that $n$ is prime. It was shown in \cite{Zappa97} that $AT(n)=rels(n)-rols(n)$ for even $n$, but this is not necessarily the case for odd $n$. An attempt in \cite{Zappa97} to prove the Alon Tarsi Conjecture for Latin squares of order $2^rc$, where $r>0$ and $c$ is an even integer for which the Alon-Tarsi conjecture is true, or $c$ is an odd integer such that the extended
Alon-Tarsi conjecture is true for $c$ and for $c+1$, was later shown to be incorrect by Glynn \cite{Glynn10}.

Another extension of Conjecture~\ref{conjAlonTarsi} to odd $n$ was recently suggested by Stones and Wanless \cite{StWan12}:
\begin{conj}\label{conjSW}
$rels(n)\neq rols(n)$ for all positive $n$.
\end{conj}
For even $n$, Conjecture~\ref{conjSW} is equivalent to Conjectures~\ref{conjAlonTarsi} and \ref{conjZappa}. For odd $n$ Conjecture~\ref{conjSW} is only known to hold up to $n=7$ (see \cite{Zappa97}). We shall see in Section~\ref{section3}, Theorem~\ref{thm2}, that the assumption that $rels(n)\neq rols(n)$ for odd $n$ yields a weak version of Conjecture~\ref{conjRota}, namely the possibility of partitioning the multiset union of the original bases into $n$ transversals, of which at least $n-1$ are bases.

% ----------------------------------------------------------------

\section{A modified version of Onn's colorful determinantal identity}\label{section2}

The following identity is due to Onn \cite{Onn97}:
\begin{prop}[\textbf{Onn's colorful determinantal identity}]\label{CDI}
Let ${}^1W,{}^2W,\ldots,{}^nW$ be $n$ square matrices of order $n$ over a field $F$. Then
\begin{equation}\label{CDIFormula}
\sum_{\rho\in S^n}\sgn(\rho)\prod_{i=1}^n \mathrm{det}\left({}^1W^{\rho_1(i)},\ldots,{}^nW^{\rho_n(i)}\right)=l(n)\prod_{j=1}^n \mathrm{det}\left({}^jW\right)
\end{equation}
\end{prop}
Where ${}^jW^i$ is the $i$th column of the matrix ${}^jW$, $S^n$ is the set of $n$-tuples over the symmetric group $S_n$ and for each such $n$-tuple $\rho=(\rho_1,\ldots,\rho_n)$, $\sgn(\rho)=\prod_{i=1}^n\sgn(\rho_i)$.

Based on (\ref{CDIFormula}) Onn argues as follows: Suppose $n$ is even and Conjecture~\ref{conjAlonTarsi} holds. If the columns of each of the matrices ${}^1W,{}^2W,\ldots,{}^nW$ form a basis and $\Text{char}(F)\nmid l(n)$  then the right hand side of (\ref{CDIFormula}) is nonzero and thus some term in the sum on the left hand side of (\ref{CDIFormula}) must be nonzero. Hence, there exists a colorful repartition of the multiset of column of the matrices ${}^iW$ consisting of bases. This implies Conjecture~\ref{conjRota} for a field of characteristic not dividing $l(n)$. For odd $n$ we know that $l(n)=0$ and thus we cannot conclude Rota's Conjecture~\ref{conjRota}. In fact, for odd $n$, the sum on the left hand side of (\ref{CDIFormula}) can be seen to be zero by a direct argument:

For any $\pi\in S_n$ and $\rho=(\rho_1,\ldots,\rho_n)\in S^n$, let $\rho\pi=(\rho_1\pi,\ldots,\rho_n\pi)$. Suppose $n$ is odd. Then,
\begin{equation}\label{eq0:1}
\begin{split}
\sgn(\rho\pi)\prod_{i=1}^n \mathrm{det} & \left({}^1W^{\rho_1\pi(i)},\ldots,{}^nW^{\rho_n\pi(i)}\right)
\\ &=\sgn(\rho)\sgn(\pi)^n\prod_{i=1}^n \mathrm{det}\left({}^1W^{\rho_1(i)},\ldots,{}^nW^{\rho_n(i)}\right)
\\ &=\sgn(\pi)\sgn(\rho)\prod_{i=1}^n \mathrm{det}\left({}^1W^{\rho_1(i)},\ldots,{}^nW^{\rho_n(i)}\right).
\end{split}
\end{equation}
It follows that each term in the sum on the left hand side of (\ref{CDIFormula}) appears $n!/2$ times with a positive sign and $n!/2$ times with a negative sign, and thus the sum on the left hand side of (\ref{CDIFormula}) is 0.

We shall modify the identity (\ref{CDIFormula}) so that the left hand side does not contain multiple terms. For this we take only elements of $S^n$ whose first component is the identity permutation. In this case the expression on the right hand side of (\ref{CDIFormula}) is divided by $n!$:
\begin{prop}[\textbf{Modified colorful determinantal identity}]\label{refinedOnn}
Let ${}^1W,{}^2W,\ldots,{}^nW$ be $n$ square matrices of order $n$ over a field. Then
\begin{equation}\label{eq7}
\sum_{\tiny\begin{array}{l}\rho\in S^n\\ \rho_1=\Text{id}\end{array}}\sgn(\rho)\prod_{i=1}^n \Text{det}\left({}^1W^i,{}^2W^{\rho_2(i)},\ldots,{}^nW^{\rho_n(i)}\right)=\frac{l(n)}{n!}\prod_{j=1}^n \Text{det}\left({}^jW\right).
\end{equation}
\end{prop}
We see that if $n$ is odd we still have zero on the right hand side and we still cannot conclude anything about Conjecture~\ref{conjRota}. However, Proposition~\ref{refinedOnn} is a first step in constructing an identity that does not vanish. We shall see in Section~\ref{section3} how Equation~(\ref{eq7}) can be modified so that the expression on the right hand side of (\ref{eq7}) becomes nonzero, thus obtaining a result for Conjecture~\ref{conjRota} for odd $n$.

\begin{proof}[Proof of Proposition~\ref{refinedOnn}] An argument similar to the one applied in (\ref{eq0:1}) shows that if $n$ is even, then every term on the right hand side of (\ref{CDIFormula}) appears $n!$ times. This implies (\ref{eq7}) for even $n$. If $n$ is odd this the result does not follow immediately.
The proof presented here mimics the proof in \cite{Onn97}, except that here $\rho_1=\Text{id}$. We use the notation ${}^{j}{W}_{k}^{l}$ for the entry in position $(k,l)$ in the matrix ${}^{j}{W}$. Let
\begin{equation}\label{eq1}
\triangle=\sum_{\tiny\begin{array}{l}\rho,\sigma\in S^n\\ \rho_1=\textrm{id}\end{array}}\sgn(\sigma)\sgn(\rho)\prod_{i,j=1}^n{}^{j}{W}_{\sigma_i(j)}^{\rho_j(i)}.
\end{equation}
We compute $\triangle$ in two different ways. For $\rho=(\Text{id},\rho_2,\ldots,\rho_n)$ let
\begin{equation*}
\begin{split}
\triangle^\rho &=\sum_{\sigma\in S^n}\sgn(\sigma)\prod_{i,j=1}^n{}^{j}{W}_{\sigma_i(j)}^{\rho_j(i)}
\\ &=\prod_{i=1}^n\sum_{\sigma_i\in S_n}\sgn(\sigma_i)\prod_{j=1}^n{}^{j}{W}_{\sigma_i(j)}^{\rho_j(i)}
\\ &=\prod_{i=1}^n \det\left({}^1W^i,{}^2W^{\rho_2(i)},\ldots,{}^nW^{\rho_n(i)}\right).
\end{split}
\end{equation*}
Applying the last equation to (\ref{eq1}) we have
\begin{equation}\label{eq5}
\begin{split}
\triangle &=\sum_{\tiny\begin{array}{l}\rho\in S^n\\ \rho_1=\textrm{id}\end{array}}\sgn(\rho)\triangle^\rho
\\ &=\sum_{\tiny\begin{array}{l}\rho\in S^n\\ \rho_1=\textrm{id}\end{array}}\sgn(\rho)\prod_{i=1}^n \det\left({}^1W^i,{}^2W^{\rho_2(i)},\ldots,{}^nW^{\rho_n(i)}\right).
\end{split}
\end{equation}
For $\sigma\in S^n$ let
\begin{equation*}
\begin{split}
\triangle_\sigma &= \sum_{\tiny\begin{array}{l}\rho\in S^n\\ \rho_1=\textrm{id}\end{array}}\sgn(\rho)\prod_{i,j=1}^n{}^{j}{W}_{\sigma_i(j)}^{\rho_j(i)}
\\ &=\left(\prod_{i=1}^n{}^{1}{W}_{\sigma_i(1)}^{i}\right)\prod_{j=2}^n\sum_{\rho_j\in S_n}\sgn(\rho_j)\prod_{i=1}^n{}^{j}{W}_{\sigma_i(j)}^{\rho_j(i)}
\\ &=\left(\prod_{i=1}^n{}^{1}{W}_{\sigma_i(1)}^{i}\right)\prod_{j=2}^n \det\left({}^jW_{\sigma_1(j)},{}^jW_{\sigma_2(j)},\ldots,{}^jW_{\sigma_n(j)}\right).
\end{split}
\end{equation*}
 Note that $\triangle_\sigma$ is nonzero only possibly for $\sigma=(\sigma_1,\ldots,\sigma_n)$ satisfying that for each $j=2,\ldots,n$, the set $\{\sigma_1(j),\ldots,\sigma_n(j)\}$ is equal to the set $\{1,\ldots,n\}$. In this case we must have that the set $\{\sigma_1(1),\ldots,\sigma_n(1)\}$ is also equal to the set $\{1,\ldots,n\}$. For each such $\sigma$ there exists $\pi_\sigma=(\pi_1,\ldots,\pi_n)\in S^n$ so that $\sigma_i(j)=\pi_j(i)$ for all $i,j=1,\ldots,n$. We have
\begin{equation}\label{eq4}
\begin{split}
\triangle_\sigma &=\left(\prod_{i=1}^n{}^{1}{W}_{\sigma_i(1)}^{i}\right)
\prod_{j=2}^n \det\left({}^jW_{\sigma_1(j)},\ldots,{}^jW_{\sigma_n(j)}\right)
\\ &=\left(\prod_{i=1}^n{}^{1}{W}_{\pi_1(i)}^{i}\right)
\prod_{j=2}^n \det\left({}^jW_{\pi_j(1)},\ldots,{}^jW_{\pi_j(n)}\right)
\\ &=\left(\prod_{i=1}^n{}^{1}{W}_{\pi_1(i)}^{i}\right)
\prod_{k=2}^n\sgn(\pi_k)\prod_{j=2}^n \det\left({}^jW\right)
\\ &=\sgn(\pi_\sigma)\prod_{j=2}^n \det\left({}^jW\right)
\sgn(\pi_1)\left(\prod_{i=1}^n{}^{1}{W}_{\pi_1(i)}^{i}\right).
\end{split}
\end{equation}
Each $\sigma$ as in (\ref{eq4}) defines a Latin square $L$ whose rows and columns are the elements of $\sigma$ and $\pi_\sigma$ respectively. thus $\sgn(L)=\sgn(\sigma)\sgn(\pi_\sigma)$. Substituting (\ref{eq4}) into (\ref{eq1}) we have:
\begin{equation}\label{eq4:1}
\begin{split}
\triangle &=\sum_{\sigma}\sgn(\sigma)\triangle_\sigma
\\ &=\sum_{\sigma}\sgn(\sigma)\sgn(\pi_\sigma)\prod_{j=2}^n \det\left({}^jW\right)\sgn(\pi_1)\left(\prod_{i=1}^n{}^{1}{W}_{\pi_1(i)}^{i}\right)
\\ &=\prod_{j=2}^n \det\left({}^jW\right)\sum_{L\in\mathcal{L}}\sgn(L)
\sgn(\pi_1)\left(\prod_{i=1}^n{}^{1}{W}_{\pi_1(i)}^{i}\right).
\end{split}
\end{equation}
Note that $\pi_1$ is the first column of the Latin square $L$. Now, instead of summing over all Latin squares and considering their first column $\pi_1$ we sum over all permutations $\pi_1$ and then over all Latin squares for which $\pi_1$ is their first column. Applying this change to (\ref{eq4:1}) we have

\begin{equation}\label{eq4:2}
\begin{split}
\triangle &=\prod_{j=2}^n \det\left({}^jW\right)\sum_{L\in\mathcal{L}}\sgn(L)
\sgn(\pi_1)\left(\prod_{i=1}^n{}^{1}{W}_{\pi_1(i)}^{i}\right)
\\ &=\prod_{j=2}^n \det\left({}^jW\right)\sum_{\pi_1\in S_n}\sgn(\pi_1)
\left(\prod_{i=1}^n{}^{1}{W}_{\pi_1(i)}^{i}\right)
\sum_{\tiny\begin{array}{l}L\in\mathcal{L}\\ \pi_1=L^1\end{array}}\sgn(L).
\end{split}
\end{equation}
We have
\begin{equation*}
\sum_{\pi_1\in S_n}\sgn(\pi_1)\left(\prod_{i=1}^n{}^{1}{W}_{\pi_1(i)}^{i}\right)=\Text{det}({}^1W),
\end{equation*}
and
\begin{equation*}
\sum_{\tiny\begin{array}{l}L\in\mathcal{L}\\ \pi_1=L^1\end{array}}\sgn(L)
\end{equation*}
is the number of even Latin squares with $\pi_1$ as their first column minus the number of odd ones. We claim that this number is equal to $l(n)/n!$ (see Lemma~\ref{lem1} part (ii) below). Assuming this, Equation~(\ref{eq4:2}) becomes
\begin{equation}\label{eq6}
\triangle =\prod_{j=2}^n \det\left({}^jW\right)\frac{l(n)}{n!}\det\left({}^1W\right)
= \frac{l(n)}{n!}\prod_{j=1}^n \det\left({}^jW\right).
\end{equation}
Combining (\ref{eq5}) and (\ref{eq6}) the result follows.
\end{proof}
To conclude this section we need a lemma. First some notation:
\begin{notn}\label{not2}
Let $\sigma,\pi\in S_n$, $i,j\in\{1,\ldots,n\}$
\begin{enumerate}
\item[\Text{(i)}] $l_{\sigma,i}(n)$ will denote the difference between the numbers of even and odd Latin squares with $\sigma$ as their $i$th row.
\item[\Text{(ii)}] $l^{\pi,i}(n)$ will denote the difference between the numbers of even and odd Latin squares with $\pi$ as their $i$th column.
\item[\Text{(iii)}] $l_{\sigma,i}^{\pi,j}(n)$ will denote the difference between the numbers of even and odd Latin squares with $\sigma$ as their $i$th row and $\pi$ as their $j$th column.
\end{enumerate}
\end{notn}
\begin{lem}\label{lem1}
Let $n$ be positive, $\pi\in S_n$, $1\leq i\leq n$
\begin{enumerate}
\item[\Text{(i)}] If $n\geq3$ is odd then $l_{\pi,i}(n)=l^{\pi,i}(n)=0$.
\item[\Text{(ii)}] For all $n$, $l_{\pi,1}(n)=l^{\pi,1}(n)=l(n)/n!$.
\end{enumerate}
\end{lem}
\begin{proof}
(i) For a given $i$, fix some $j$ and $k$ different from $i$. For any Latin square with $\pi$ as its $i$th row we can obtain a Latin square with the opposite parity by exchanging the $j$th and $k$th rows. Similarly for columns.

(ii) If $n$ is odd then $l_{\pi,1}(n)=l^{\pi,1}(n)=l(n)=0$ by part (i). Suppose $n$ is even and let $\pi_1,\pi_2\in S_n$. Let $L$ be a Latin square containing $\pi_1$ as its $i$th row. By applying the permutation $\pi_2\circ\pi_1^{-1}$ on the rows of $L$ we obtain a Latin square with the same parity containing $\pi_2$ as its $i$th row. Thus $l_{\pi_1,1}(n)=l_{\pi_2,1}(n)=l(n)/n!$, and similarly for columns.
\end{proof}
% ----------------------------------------------------------------

\section{A weak case of Rota's basis conjecture for odd $n$}\label{section3}
We saw in Section~\ref{section2} that the colorful determinantal identity cannot be used to conclude anything about Rota's basis conjecture for odd dimensions. In this section we shall invert the signs of half of the terms in the sum on the left hand side of Equation~(\ref{eq7}) so that the sum on the right hand side will not vanish.

Recall Notation \ref{not2}. We have
\begin{lem}\label{lem2}
Let $\sigma, \pi \in S_n$.
\begin{enumerate}
\item[\Text{(i)}] if $n$ is even then $l_{\sigma,1}^{\pi,1}(n)=rels(n)-rols(n)$.
\item[\Text{(ii)}] if $n$ is odd then $l_{\sigma,1}^{\pi,1}(n)=\sgn(\sigma)\sgn(\pi)(rels(n)-rols(n))$.
\end{enumerate}
\end{lem}
Before proving the lemma, recall from \cite{Janssen95} that an \emph{isotopy} is a triple $(\alpha,\beta,\gamma)$  such that $\alpha,\beta,\gamma\in S_n$ and it acts on a Latin square $L$ by applying $\alpha$ on the set of rows, $\beta$ on the set of columns, and $\gamma$ on the symbols of the square.
\begin{proof}[Proof of Lemma~\ref{lem2}]
Let $L$ be a Latin square containing $\sigma$ as its first row and $\pi$ as its first column. If $L_{1,1}=k\neq1$, we apply the inversion $\gamma=(1,k)$ on the set $\{1,\ldots,n\}$ to obtain a square with 1 as its $(1,1)$ entry (If $L_{1,1}=1$ then $\gamma=\Text{id}$). Now apply a permutation  $\alpha$ on the rows and a permutation $\beta$ on the columns to obtain a reduced Latin square. Since $\alpha$ and $\beta$ are determined by $\pi$ and $\sigma$ respectively, the isotopy $(\alpha,\beta,\gamma)$ can be applied on any square containing $\sigma$ as its first row and $\pi$ as its first column, to obtain a reduced square. Now $\alpha$ and $\beta$ have the same parity if and only if $\sigma$ and $\pi$ have the same parity (since $\alpha\circ\gamma=\pi^{-1}$ and $\beta\circ\gamma=\sigma^{-1}$). According to Proposition 3.1 in \cite{Janssen95} the resulting square and the original square have opposite parities only in the case that $n$ is odd and $\sgn(\sigma)=-\sgn(\pi)$.
\end{proof}

\begin{thm}\label{thm1}
Let ${}^1W,{}^2W,\ldots,{}^nW$ be $n$ square matrices of order $n$ over a field, where $n$ is odd. Then,
\begin{equation}\label{eq8}
\begin{split}
\sum_{\tiny\begin{array}{l}\rho\in S^n\\ \rho_1=\Text{id}\end{array}}\sgn(\rho)
\per\left({}^1W^1,{}^2W^{\rho_2(1)},\ldots,{}^nW^{\rho_n(1)}\right)
\prod_{i=2}^n\Text{det}\left({}^1W^i,{}^2W^{\rho_2(i)},\ldots,{}^nW^{\rho_n(i)}\right)
\\ =(n-1)!\cdot (rels(n)-rols(n))\per\left({}^1W\right)\prod_{j=2}^n \Text{det}\left({}^jW\right),
\end{split}
\end{equation}
where $\per(W)$ denotes the permanent of the matrix $W$.
\end{thm}
\begin{proof}
On the left hand side of (\ref{eq7}) we have a sum of $(n!)^{n-1}$ terms each consisting of the product of $n$ determinants. If we omit the signs in the first determinant of each such term we get the product of $n-1$ determinants and one permanent. We shall see that this can be achieved by omitting $\sgn(\sigma_1)$ in (\ref{eq1}). We denote the resulting expression  by $\triangle^{\prime}$ instead of $\triangle$ and Equation~(\ref{eq1}) takes the following form:
\begin{equation}\label{eq2_0}
\begin{split}
\triangle^{\prime} &=\sum_{\tiny\begin{array}{l}\rho,\sigma\in S^n\\ \rho_1=\textrm{id}\end{array}}
\frac{\sgn(\sigma)}{\sgn(\sigma_1)}\sgn(\rho)\prod_{i,j=1}^n{}^{j}{W}_{\sigma_i(j)}^{\rho_j(i)}
\\ &=\sum_{\tiny\begin{array}{l}\rho,\sigma\in S^n\\ \rho_1=\textrm{id}\end{array}}
\sgn(\sigma_1)\sgn(\sigma)\sgn(\rho)\prod_{i,j=1}^n{}^{j}{W}_{\sigma_i(j)}^{\rho_j(i)}.
\end{split}
\end{equation}
We can compute $\triangle^{\prime}$ in two different ways. Taking the external sum in (\ref{eq2_0}) by $\rho$ we obtain
\begin{equation}\label{eq2_1}
\begin{split}
\triangle^{\prime} &=\sum_{\tiny\begin{array}{l}\rho\in S^n\\ \rho_1=\Text{id}\end{array}}\sgn(\rho)
\sum_{\sigma\in S^n}\sgn(\sigma_1)\sgn(\sigma)\prod_{i,j=1}^n{}^{j}{W}_{\sigma_i(j)}^{\rho_j(i)}
\\ &=\sum_{\tiny\begin{array}{l}\rho\in S^n\\ \rho_1=\Text{id}\end{array}}\sgn(\rho)
\left(\sum_{\sigma_1\in S_n}\prod_{j=1}^n{}^{j}{W}_{\sigma_1(j)}^{\rho_j(1)}\right)
\left(\prod_{i=2}^n\sum_{\sigma_i\in S_n}\sgn(\sigma_i)\prod_{j=1}^n{}^{j}{W}_{\sigma_i(j)}^{\rho_j(i)}\right)
\\ &=\sum_{\tiny\begin{array}{l}\rho\in S^n\\ \rho_1=\Text{id}\end{array}}\sgn(\rho)
\per\left({}^1W^1,{}^2W^{\rho_2(1)},\ldots,{}^nW^{\rho_n(1)}\right)
\prod_{i=2}^n\Text{det}\left({}^1W^i,{}^2W^{\rho_2(i)},\ldots,{}^nW^{\rho_n(i)}\right).
\end{split}
\end{equation}
Taking the external sum in (\ref{eq2_0}) by $\sigma$ and substituting $\triangle_\sigma$ from (\ref{eq4}) we obtain
\begin{equation}\label{eq2_2}
\begin{split}
\triangle^\prime &=\sum_{\sigma}\sgn(\sigma_1)\sgn(\sigma)\triangle_\sigma
\\ &=\sum_{\sigma}\sgn(\sigma_1)\sgn(\sigma)\sgn(\pi_\sigma)\prod_{j=2}^n \det\left({}^jW\right)\sgn(\pi_1)\left(\prod_{i=1}^n{}^{1}{W}_{\pi_1(i)}^{i}\right).
\end{split}
\end{equation}
In a similar manner as (\ref{eq4:1}) was derived from (\ref{eq4}) equation (\ref{eq2_2}) yields
\begin{equation}\label{eq2_2:1}
\triangle^\prime =\prod_{j=2}^n \det\left({}^jW\right)\sum_{L\in\mathcal{L}}\sgn(\sigma_1)\sgn(L)
\sgn(\pi_1)\left(\prod_{i=1}^n{}^{1}{W}_{\pi_1(i)}^{i}\right),
\end{equation}
where $\sigma_1$ is the first row of $L$ and $\pi_1$ is the first column. Now, instead of summing over all Latin squares, we sum over all possible first columns, then over all possible first rows, with the given first column, and then over all Latin squares with these first row and first column. The sum in (\ref{eq2_2:1}) becomes
\begin{equation}\label{eq2_2:2}
\triangle^\prime =\prod_{j=2}^n \det\left({}^jW\right)
\sum_{\pi_1\in S_n} \sgn(\pi_1)\left(\prod_{i=1}^n{}^{1}{W}_{\pi_1(i)}^{i}\right)
\sum_{\tiny\begin{array}{c}\sigma_1\in S_n\\ \pi_1(1)=\sigma_1(1)\end{array}}\sgn(\sigma_1)
\sum_{\tiny\begin{array}{c}L\in\mathcal{L}\\ \sigma_1=L_1\\ \pi_1=L^1\end{array}}\sgn(L).
\end{equation}
For any $\pi_1,\sigma_1\in S_n$ we have, by Lemma~\ref{lem2},
\begin{equation*}
\sgn(\pi_1)\sgn(\sigma_1)
\sum_{\tiny\begin{array}{c}L\in\mathcal{L}\\ \sigma_1=L_1\\ \pi_1=L^1\end{array}}\sgn(L)=rels(n)-rols(n).
\end{equation*}
Also, for each $\pi_1\in S_n$ there are $(n-1)!$ permutations $\sigma_1\in S_n$ satisfying $\pi_1(1)=\sigma_1(1)$. Applying this to (\ref{eq2_2:2}) gives
\begin{equation}\label{eq2_3}
\begin{split}
\triangle^\prime &=\prod_{j=2}^n \det\left({}^jW\right)(n-1)!\cdot (rels(n)-rols(n))\sum_{\pi_1\in S_n}\left(\prod_{i=1}^n{}^{1}{W}_{\pi_1(i)}^{i}\right)
\\ &=\prod_{j=2}^n \det\left({}^jW\right)(n-1)!\cdot (rels(n)-rols(n))\per\left({}^1W\right).
\end{split}
\end{equation}
Combining (\ref{eq2_1}) and (\ref{eq2_3}) the result follows.
\end{proof}
We can now obtain a weak version of Rota's basis conjecture for odd $n$:
\begin{thm}\label{thm2}
Suppose $n$ is odd and $rels(n)\neq rols(n)$. If $B_1,B_2,\ldots,B_n$ are bases of a vector space of dimension $n$ over a field of characteristic 0, then their multiset union can be partitioned into $n$ transversals, such that at least $n-1$ of them are bases.
\end{thm}
\begin{proof}
For each $j=1,\ldots,n$ let ${}^jW$ be the matrix whose columns are the elements of $B_j$. If $\per({}^1W)\ne 0$, then the right hand side of (\ref{eq8}) is nonzero. It follows that at least one of the terms in the sum on the left hand side of (\ref{eq8}) is nonzero. This term gives $n-1$ transversals that are bases.

It remains to show that the assumption that $\per({}^1W)\ne 0$ is not needed.
It will be shown that if $\det({}^1W)\ne 0$ then it is possible to perform a sequence of row operations on all matrices ${}^1W,\dots,{}^nW$ simultaneously, to obtain matrices ${}^1W',\dots,{}^nW'$ such that $\per({}^1W')\ne 0$. The validity of the theorem is not affected by simultaneous row operations, since the inverse of these row operations can be applied to the transversals obtained from ${}^1W',\dots,{}^nW'$, to yield transversals for $B_1,B_2,\ldots,B_n$.

To define these row operations let $A={}^1W$. Assume that $\per(A)=0$ and let $k<n$ be maximal such that there exists a $k\times k$ sub-matrix of $A$ having nonzero permanent and determinant. Without loss of generality we may assume that this sub-matrix is $A[1,\ldots,k \mid 1,\ldots,k]$.

By exchanging some row $i$, ($i>k$) with row $k+1$, if necessary, we may assume that the determinant of $A_{k+1}:=A[1,\ldots,k+1 \mid 1,\ldots ,k+1]$ is nonzero, otherwise the rank of the first $k+1$ columns is $k$, contradicting the assumption that $\det(A)\ne 0$. By the maximality of $k$ we have $\per(A_{k+1})= 0$.

Let $\vec{u}$ be the $k+1$-vector whose $i$-th coordinate is $per(A[1,\ldots,k  \mid 1,\ldots ,\hat{i},\ldots ,k+1])$. By our assumption on $k$ this is not the zero vector, since its $k+1$ coordinate is $A[1,\ldots,k \mid 1,\ldots,k]$.

Now, for each $i \le k$, let $R_i[k+1]$ be the $i$-th row of $A_{k+1}$ and let $A_{k+1}^{i}$ be the matrix obtained from $A_{k+1}$ by adding the row $R_i[k+1]$ to the row $R_{k+1}[k+1]$. Then, $\per(A_{k+1}^{i})=R_i[k+1]\cdot \vec{u}+\per(A_{k+1})=R_i[k+1]\cdot \vec{u}$ (since we assumed $\per(A_{k+1})= 0$).

If $R_i[k+1]\cdot \vec{u}=0$ for all $i\le k$, then all $R_i[k+1]$ are perpendicular to $\vec{u}$. Since we are assuming that $per(A[1,\ldots,k+1 \mid 1,\ldots,k+1])= 0$,  also $R_{k+1}[k+1]$ is perpendicular to $\vec{u}$. Thus, all the rows of $A_{k+1}$ are perpendicular to $\vec{u}$, contrary to the assumption that $\det(A_{k+1})\ne 0$. Hence, it is possible to perform a row operation on $A_{k+1}$ to obtain a matrix with nonzero permanent. It follows that it is possible to perform a row operation on $A$ to obtain a matrix containing a $(k+1)\times(k+1)$ sub-matrix having nonzero determinant and permanent. Proceeding this way, we obtain an $n\times n$ matrix with nonzero determinant and permanent.
\end{proof}

% ----------------------------------------------------------------
\section{A `scrambled' version of Rota's conjecture}

Suppose that before trying to find a decomposition into $n$ independent transversals, we scramble the matrices. Is Rota's conjecture still true?
To put this question formally, call a set  $S_1,\ldots,S_n$ of subsets of $\mathbb{R}^n$, each of size $n$, a {\em scrambled system of bases} if $\bigcup S_i$ is the multiset union of $n$ bases. By Rado's theorem, a scrambled system of bases has at least one independent transversal. The question is - does $\bigcup S_i$ necessarily decompose into  $n$ independent transversals?

The answer is ``no'', at least for $n$ odd.

\begin{example}
Let $n$ be odd. Let $S_i~(i=1,\ldots,n)$ consist of $n-1$ copies of the standard basis vector $e_i$, and one copy of $e_{i+1}-e_{i+2}$, where
indices are taken $\bmod n$.
\end{example}

The decomposition into the bases $B_i=\{e_i, e_{i+1}, \ldots, e_{i+n-2}, e_{i+n-1}-e_{i+n-2}\}$ ($i \le n$, indices taken $\bmod n$) shows that this is a scrambled system of bases.

Suppose that $\bigcup S_i$ is decomposed into $n$ linearly independent transversals. The sum of the non-$e_i$ vectors is zero, and hence the set consisting of the non-$e_i$'s is not one of these transversals. Hence there exist at least two transversals containing both $e_i$'s and non-$e_i$'s. Since $n$ is odd, at least one of these transversals contains fewer than $\frac{n}{2}$ non-$e_i$'s. But this entails that this transversal contains vectors $e_i,e_{i+1},e_i-e_{i+1}$, meaning that this transversal is not linearly independent.

In \cite{AB06} it was proved (in a more general setting, of matroids) that a scrambled system of bases can be covered by $2n$ independent transversals, and it was conjectured that in fact $n+1$ independent transversals suffice. Here we suggest that a result similar to the one proved in this paper is true for scrambled systems of bases:

\begin{conj}
A  scrambled system of bases of $\mathbb{R}^n$ contains $n-1$  disjoint independent transversals.
\end{conj}
{\bf Acknowledgment:} The authors are indebted to Martin Loebl and Ran Ziv for fruitful discussions.
% ----------------------------------------------------------------

\end{document}